\documentclass[11pt]{amsart}

\usepackage{amsmath}
\usepackage{amssymb}
\usepackage{amsthm}
\usepackage{graphicx}
\usepackage{tikz}
\usepackage{xcolor}
\usepackage{hyperref}

\allowdisplaybreaks

\usepackage[margin=1.2in]{geometry}
\usepackage{bookmark}

\def\E{\mathbb{E}}
\def\var{\mathbb{Var}}

\def\po{{\rm Po}}

\def\deg{{\rm deg}}

\def\E{\mathbb{E}}
\def\R{\mathbb{R}}

\def\Z{\mathbb{Z}}

\def\b{\beta}
\def\eps{\varepsilon}

\def\cC{\mathcal {C}}

\def\cP{\mathcal {P}}
\def\cT{\mathcal {T}}
\def\cO{\mathcal {O}}
\def\cE{\mathcal {E}}
\def\cI{\mathcal {I}}

\def\1{\mathbf{1}}

\def\lam {\lambda}

\def\tce{t_c + \eps}
\def\tce2{t_c + \frac{\eps}{2}}

\def\var{\text{var}}

\newtheorem*{theorem*}{Theorem}
\newtheorem{theorem}{Theorem}
\newtheorem{lemma}[theorem]{Lemma}
\newtheorem{cor}[theorem]{Corollary}
\newtheorem{defn}[theorem]{Definition}
\newtheorem*{defn*}{Definition}

\newtheorem*{prop*}{Proposition}

\newtheorem*{conj*}{Conjecture}

\newtheorem*{fact*}{Fact}

\begin{document}
\title{Independent sets of a given size and structure in the  hypercube}

\author{Matthew Jenssen}
\address{School of Mathematics\\ University of Birmingham}
\author{Will Perkins}
\address{Department of Mathematics, Statistics, and Computer Science\\ University of Illinois at Chicago}
\author{Aditya Potukuchi}
	\email{m.jenssen@bham.ac.uk \\ math@willperkins.org \\adityap@uic.edu }

\date{\today}

\begin{abstract}
We determine the asymptotics of the number of independent sets of size $\lfloor \beta 2^{d-1} \rfloor$ in the discrete hypercube $Q_d = \{0,1\}^d$ for any fixed $\beta \in (0,1)$ as $d \to \infty$, extending a result of Galvin for $\beta \in (1-1/\sqrt{2},1)$.  Moreover, we prove a multivariate local central limit theorem for structural features of independent sets in $Q_d$ drawn according to the hard core model at any fixed fugacity $\lam>0$.  In proving these results we develop several general tools for performing combinatorial enumeration using polymer models and  the cluster expansion from statistical physics along with local central limit theorems.
\end{abstract}

\maketitle

\section{Introduction}
\label{secIntro}

Let $Q_d$ be the discrete hypercube:\ the graph with vertex set $\{0,1\}^d$ in which two vectors are joined by an edge if they differ in exactly one coordinate.  An independent set is a set of vertices that contains no edge.  Let $\cI(Q_d)$ be the set of independent sets of $Q_d$ and let $i(Q_d)= | \cI(Q_d)|$ be the number of independent sets of the hypercube.   The vertices of $Q_d$ can be divided into two sets, those whose coordinates sum to an even number and those whose coordinates sum to an odd number.  This partition shows that $Q_d$ is a bipartite graph.  We let $N := 2^{d-1}$ be the number of even (or odd) vertices of $Q_d$.  A trivial lower bound on $i(Q_d)$ is $2 \cdot 2^{N} -1$ obtained by considering independent sets of only even or only odd vertices.   A better lower bound is obtained by considering independent sets with an arbitrary (but constant) number of `defect' vertices on one side of the bipartition.  This increases the lower bound by a factor $\sqrt{e}$.   Korshunov and Sapozhenko showed that this gives the  correct asymptotics for $i(Q_d)$ as $d \to \infty$~\cite{korshunov1983number}. 
\begin{theorem}[Korshunov and Sapozhenko]
\label{thmSapKors}
As $d \to \infty$,
\[ i(Q_d) = (2\sqrt{e} +o(1)) 2^{N} \,.\]
\end{theorem}

Galvin later studied weighted independent sets in the hypercube.  For $\lam \ge 0$, define the independence polynomial of $Q_d$, 
\begin{align*}
Z_{Q_d} (\lam) = \sum_{I \in \cI(Q_d)} \lam ^{|I| } \,.
\end{align*}
Taking $\lam=1$ recovers $i(Q_d)$.
In what follows we will drop $Q_d$ from the notation, writing $Z(\lam)$ and $\cI$ for $Z_{Q_d}(\lam)$ and $\cI(Q_d)$.    

The independence polynomial $Z(\lam)$  is also the partition function of the hard-core model on $Q_d$: the probability distribution $\mu_\lam$ on $\cI$ defined by $\mu_\lam(I) = \lam^{|I|}/Z(\lam)$.     By generalizing Sapozhenko's alternative proof of Theorem~\ref{thmSapKors} in~\cite{sapozhenko1987number}, Galvin found the asymptotics for $Z(\lam)$ for $\lam > \sqrt{2}-1$~\cite{galvin2011threshold} (as well as the asymptotics of $\log Z(\lam)$ for $\lam = \Omega(\log d/d^{1/3})$).
\begin{theorem}[Galvin]
\label{thmGalvinZ}
For $\lam > \sqrt{2}- 1$, 
\[ Z(\lam) = (2+o(1)) (1+\lam)^{N} \exp \left[ \lam N  \left( \frac{1}{1+\lam}  \right)^d  \right ]    \]
as $d \to \infty$. 
\end{theorem}
Analogously to  Theorem~\ref{thmSapKors}, the trivial lower bound for $Z(\lam)$ is $2 (1+\lam)^N -1$, and  the asymptotic formula in Theorem~\ref{thmGalvinZ} includes the contribution from independent sets with a constant number of defect vertices, captured by the exponential factor. 

Galvin also studied the typical structure of the defect vertices under the probability distribution $\mu_{\lam}$. Formally, given an independent set $I\in \cI$, if $|I\cap\cO|\leq |I\cap\cE|$, we refer to the elements of $I\cap\cO$ as the \emph{defect} vertices of $I$; otherwise we say that $I\cap\cE$ is the set of defect vertices.  A natural way to describe the structure of a set $S\subset\cO, \cE$ is to describe the graph $Q_d^2[S]$ where $Q_d^2$ denotes the square of $Q_d$. Given an independent set $I$ with defect vertices $S$, we refer to the connected components of $Q_d^2[S]$ as the \emph{defects} of $I$. Galvin showed that for $\lam>\sqrt{2}-1$, all but a  vanishing fraction of $Z(\lam)$ comes from independent sets with defects  of size at most $1$. 

Recently, the first two authors found the asymptotics of $Z(\lam)$ for all fixed $\lam>0$~\cite{jenssen2020independent}.  The asymptotic formula takes into account defects of arbitrary, but constant size.  The smaller $\lam$ is, the larger the size of defects that must be considered. 
\begin{theorem}[Jenssen and Perkins]
\label{thmJPZ}
There is a sequence of polynomials $R_j(d,\lam)$, $j\in \mathbb N$, such that for any fixed $t\ge 1$ and $\lam > 2^{1/t} -1$, 
\[ Z(\lam) = (2+o(1)) (1+\lam)^N \exp \left[ N  \sum_{j=1}^{t-1} R_j(d,\lam) (1+\lam)^{-dj}  \right ]  \]
as $d \to \infty$. 
Moreover the coefficients of the polynomial $R_j$ can be computed in time $e^{O( j \log j)}$.
\end{theorem}
In particular, $R_1 =\lam$, recovering the formula in Theorem~\ref{thmGalvinZ}. 

Given these results it is natural to ask for the asymptotics of $i_m(Q_d)$,   the number of independent sets of size $m$ in $Q_d$.   There is a trivial lower bound of $i_m(Q_d) \ge 2 \binom{N}{m}$, obtained by considering independent sets composed entirely of even or odd vertices, but depending on how large $m$ is, we may need to take into account independent sets of size $m$ with defects up to a  given size.  

Galvin~\cite{galvin2012independent} gave the  asymptotics of $i_m(Q_d)$ in the range for which almost all independent sets of size $m$ contain defects of size at most $1$. 

\begin{theorem}[Galvin]
\label{thmGalvinK}
Fix  $\beta \in (1-1/\sqrt{2}, 1)$ and let  $\lam = \frac{\beta}{1-\beta}$.  Then
\begin{align}
\label{eqGalvin1}
 i_{\lfloor \beta N \rfloor}(Q_d) &= (2+o(1)) \binom{N}{ \lfloor \beta N \rfloor} \exp \left [ \lam N  \left(\frac{1}{1+\lam} \right)^{d}  \right] \,. 
\end{align}
as $d \to \infty$. 
\end{theorem}
Note that the asymptotic formula~\eqref{eqGalvin1} consists of the trivial lower bound $2 \binom{N}{ \lfloor \beta N \rfloor}$ multiplied by the same exponential correction factor in the asymptotic formula for $Z(\lam)$ in the range $\lam > \sqrt{2}-1$.   

We  show that a similar, but more complicated, formula holds for all $\beta >0$.  In particular,  when $\beta< 1- 1/\sqrt{2}$ the formula is \textit{not} simply the trivial bound multiplied by the appropriate exponential correction factor from Theorem~\ref{thmJPZ}.  We explain below where the extra complexity arises. 
\begin{theorem}
\label{thmFixedSize}
There is a sequence of rational functions $P_j(d,\beta)$, $j\in\mathbb N$, such that 
so that for any fixed $t\ge 1$ and $\beta  \in (1-2^{-1/t},1)$,
\begin{align}
 i_{\lfloor \beta N \rfloor}(Q_d) &= (2+o(1)) \binom{N}{ \lfloor \beta N \rfloor}  \exp \left[ N \sum_{j=1}^{t-1} P_j (d,\beta) \cdot (1-\beta)^{jd} \right] 
  \end{align}
 as $d \to \infty$. Moreover the coefficients of $P_j$ can be computed in time $e^{O( j \log j)}$.
 \end{theorem}

 For small values of $j$ the functions $P_j$ can be computed by hand.  For example, $P_1 = \frac{\beta}{1-\beta}$, and taking $t=2$ in Theorem~\ref{thmFixedSize} recovers Galvin's Theorem~\ref{thmGalvinK}.   A more involved calculation carried out in Section~\ref{secGivenSize} yields

\begin{align*}
P_2= \frac{ d(d - 1)  ( 2- \b ) \b^3- 2(1- \b)^2 \b^2}{4(1-\b)^4}-\frac{\beta(1-d\b)^2}{2(1-\b)^3}\, .
\end{align*}

By Theorem~\ref{thmFixedSize} this gives an explicit asymptotic formula for $i_{\lfloor \beta N \rfloor}(Q_d)$ for  $\beta>1-2^{-1/3}$:

{\scriptsize
\begin{align*}
 i_{\lfloor \beta N \rfloor}(Q_d)
  \sim 2 \binom{N}{ \lfloor \beta N \rfloor} \exp \left[ N \frac{\beta}{1-\beta} (1-\beta)^d  
 + N \left( \frac{ d(d - 1)  ( 2- \b ) \b^3- 2(1- \b)^2 \b^2}{4(1-\b)^4}-\frac{\beta(1-d\b)^2}{2(1-\b)^3} \right)  (1-\beta)^{2d}  \right ] \,.
\end{align*}
}

In principle, one can continue to compute $P_3, P_4, \dots$ and obtain explicit asymptotics for any fixed $\beta$.   More generally, the results of~\cite{jenssen2020independent} and of this paper hold for much smaller $\lam$ and $\beta$, tending to $0$ as $d \to \infty$, as long as $\lam \ge C \log d/d^{1/3}$ and $\beta > C \log d /d^{1/3}$ for an absolute constant $C$.  In this case, however, the asymptotic formulas in Theorem~\ref{thmJPZ} and Theorem~\ref{thmFixedSize} become series with a number of terms that grows with $d$.  These series can be used to give an algorithm to approximate $Z(\lam)$ and $i_{\lfloor \beta N \rfloor}(Q_d)$ up to a $(1+\eps)$ multiplicative factor in time polynomial in $1/\eps $ and $N$ (an FPTAS in the language of approximate counting; see e.g.~\cite{JenssenAlgorithmsJ} for such an algorithm for independent sets in expander graphs).  This raises an interesting question of what it means to determine the asymptotics of a sequence $f(d)$ as $d \to \infty$.  Evaluating a closed-form expression involving, say, exponentials or logarithms, might also  involve truncating a power series,  and so in a sense an algorithmic definition is natural.  We do not pursue this further here and instead stick with $\beta$ constant.

The proof of Theorem~\ref{thmFixedSize} makes use of the following simple yet  useful identity.  Let $\cI_m = \{ I \in \cI : | I | = m\}$ so that $i_m(Q_d) = | \cI_m|$.  For $m\in\mathbb N$ and $\lam>0$,
\begin{equation}
\label{eqIkIdentity}
i_m(Q_d)=\frac{Z(\lam)}{\lam^m} \mu_\lam( \cI_m) \,.
\end{equation}
In fact this formula follows from the definition of $\mu_\lam$ and so holds for any graph, not just $Q_d$. To use~\eqref{eqIkIdentity} along with Theorem~\ref{thmJPZ}  to derive asymptotics for $i_m(Q_d)$, we must compute the asymptotics of $ \mu_\lam( \cI_m)$.  The feasibility of doing this depends very much  on $m$ and  the choice of $\lam$.  By choosing $\lambda$ so that the expected size of an independent set drawn from $\mu_{\lam}$ is approximately $m$, we can compute  the asymptotics of $ \mu_\lam( \cI_m)$  using a local central limit theorem.  In practice, we do not work with the hard-core model directly, but with an approximating measure derived from a polymer model which describes the distribution on defects in an independent set from the hard-core model (see Section~\ref{secPolymers}). This polymer model was  introduced in \cite{jenssen2020independent}. 

In the next theorem we give an expansion in $(1-\beta)^d$ for a value of the fugacity $\lam$ for which the expected size of $\mathbf I$, a random sample from the hard-core model, is close to $\lfloor \beta N \rfloor$.   By expanding the formula~\eqref{eqIkFormula} below and combining this with~\eqref{eqLamBetaDef} we obtain Theorem~\ref{thmFixedSize}.
\begin{theorem}
\label{thmFixedSizelam}
There exists a sequence of rational functions $B_j(d,\beta)$, $j\in \mathbb N$, such that the coefficients of $B_j$ can be computed in time $e^{O( j \log j)}$ and the following holds.
Fix $\beta  \in (0,1)$ and let $t\ge 1$ be such that $\beta> 1-2^{-1/t}$ and let $r=\lceil t/2 \rceil-1$.  Then with
\begin{align}
\label{eqLamBetaDef}
\lam_{\beta} &= \frac{\beta}{1-\beta} + \sum_{j=1}^{r} B_j (d,\beta) \cdot (1-\beta)^{jd}
\end{align}
we have
\begin{align}
\label{eqIkFormula}
 i_{\lfloor \beta N \rfloor}(Q_d) 
= \frac{ 1+o(1)}{    \sqrt{2\pi N \beta (1-\beta)} }
  \frac{Z(\lam_{\beta})}{\lam_{\beta}^{\lfloor \beta N \rfloor  } } 
  \end{align}
 as $d \to \infty$.
 Moreover,
 \[ \left|  \E _{\lam_{\beta}}  | \mathbf I | -  \lfloor \beta N \rfloor \right|   = o(N^{1/2}) \,. \]
 \end{theorem}

In~\cite{jenssen2020independent} the authors prove a multivariate central limit theorem for the number of defects of different types in the polymer model.  In Section~\ref{secLCLT}, we establish a multivariate \emph{local} central limit theorem for this polymer model which allows us to refine Theorems~\ref{thmFixedSize} and~\ref{thmFixedSizelam} further still. Given a defect $S$, we define the \emph{type} of $S$  to be the isomorphism class of the graph $Q_d^2[S]$. For a given defect type $T$, we let $X_{T}$ be the random variable that counts the number of defects of type $T$ in a sample from the hard-core model on $Q_d$.  We let $m_{T} = \E_{\lam} X_{T}$.  
 
 Given a collection $\cT$ of types and vector of non-negative integers $ \mathbf k=(k_T )_{T \in \mathcal T}$, let $ i_{m, \mathbf x}(Q_d) $ denote the number of independent sets in $Q_d$ of size $m$ with exactly  $ k_T $ defects of type $T$ for all $T \in \mathcal T$.
 
 \begin{theorem}
\label{thmFixedSizelocal}
Fix $\beta  \in (0,1)$ and  let $\lam=\lam_\beta$ be as in Theorem~\ref{thmFixedSizelam}.  Let $\cT_1$ be the set of defect types  $T$ such that $m_T \to \rho_T$ for some constant $\rho_T >0$ as $d \to \infty$ and  $\cT_2$ the set of defect types $T$ so that $m_T\to\infty$.
Let $ ( k_T )_{T \in \mathcal T_1}$ be a vector of fixed non-negative integers and let $ ( k_T )_{T \in \mathcal T_2}$ be such that $k_T=  \lfloor m_T  +s_T \rfloor$ where $| s_T| = O(\sqrt{ m_T})$ for all $T \in \mathcal T_2$. Let $\mathbf k= (k_T)_{T\in \cT_1\cup\cT_2}$. Then
\begin{align}
 i_{\lfloor \beta N \rfloor, \mathbf k}(Q_d) 
= \frac{ 1 +o(1) }{    \sqrt{2\pi N \beta (1-\beta)} }
  \frac{Z(\lam_{\beta})}{\lam_{\beta}^{\lfloor \beta N \rfloor  } } 
  \prod_{T \in  \cT_1} \frac{\rho_T^{k_T} e^{-\rho_T}}{(k_T)!}
  \prod_{T \in \mathcal T_{2}} \frac{e^{ - \frac{ s_T^2}{2 m_T}}}{\sqrt{2 \pi m_T}}\, 
  \end{align}
 as $d \to \infty$.
  \end{theorem}
 
This formula matches that of~\eqref{eqIkFormula} with additional factors corresponding to Poisson probabilities (for $T \in \cT_1$) and local central limit theorem probabilities (for $T \in \cT_2$).

\subsection{Methods: maximum entropy, statistical mechanics, and local central limit theorems} 

The methods we use here combine several different probabilistic tools, including abstract polymer models and the cluster expansion, large deviations, and local central limit theorems.  
Counting independent sets in the hypercube is a canonical combinatorial enumeration problem, and so we hope this provides a template for using this combination of tools in  other combinatorial problems. 

There is a long history of using local central limit theorems in combinatorics, with many examples in analytic combinatorics and the study of integer partitions (see e.g.~\cite{romik2005partitions,desalvo2016robust,melczer2020counting,mckinley2020maximum}) as well as the enumeration of contingency tables~\cite{barvinok2012asymptotic} and graphs with prescribed degree sequences (e.g.~\cite{barvinok2013number,isaev2018complex}).  Here we show that local central limit theorems work very well in combination with two tools from statistical physics, polymer models and the cluster expansion, which have been used recently in combinatorial enumeration~\cite{jenssen2020independent,balogh2021independent,jenssen2020homomorphisms}.  

The connection between these methods starts with a general approach to counting via probability and the principle of maximum entropy which is laid out explicitly by Barvinok and Hartigan in~\cite{barvinok2010maximum} (and later discussed in~\cite{mckinley2020maximum}), but appears implicitly in other enumerations methods, such as the circle method (see~\cite{isaev2018complex} for an explanation of these connections).  The main idea is that to count a subset of objects defined by a number of constraints, one  considers the maximum entropy distribution on the larger set that satisfies the constraints in expectation.  The size of the subset can then be expressed as the exponential of the entropy of this distribution times the probability that a random object drawn from this distribution satisfies the constraints.  In the example of enumerating integer partitions, these maximum entropy distributions take the form of sequences of independent geometric random variables with different means~\cite{fristedt1993structure,arratia1994independent}, and asymptotic enumeration can be accomplished by solving a convex optimization problem to find these means and then proving a local central limit theorem for linear combinations of independent geometric random variables~\cite{romik2005partitions,desalvo2016robust,melczer2020counting,mckinley2020maximum}.

This approach naturally leads to considering statistical physics models.  For example, the maximum entropy distribution over independent sets in a graph with a given mean size is the hard-core model.  The entropy of the hard-core model is the log partition function minus the expected size of an independent  set:  $H(\mu_\lam) = \log Z(\lam) - \log \lam \cdot \E_{\mu} | \mathbf I|$, and so the enumeration problem for independent sets of a given size reduces to computing $\log Z$ and computing $\mu_{\lam}(\cI_k)$ (via, say, a local central limit theorem) as described above. 

The complication is that the quantities of interest (say, the size of an independent set from the hard-core model) can no longer be written as the sum of independent random variables.  When interactions are weak enough (or density small enough) correlations between vertices decay exponentially in distance and methods like the cluster expansion can be used to prove both central limit theorems and local central limit theorem. This type of result is closely related to the concept of equivalence of ensembles between the grand canonical ensemble (fixed mean energy) and the canonical ensemble (fixed energy).  For instance, Dobrushin and Tirozzi showed that for spin models with finite-range interactions on $\Z^d$, a central limit theorem implies a local central limit theorem~\cite{dobrushin1977central} (see also~\cite{campanino1979local} for an extension to long-range interactions).

  What we do here is prove local central limit theorems conditioned on a phase in the strong interaction, phase coexistence regime.  This, in combination with using polymer models and the cluster expansion to find the asymptotics of $Z(\lam)$, allows us to enumerate independent sets of a given size and structure.

In Section~\ref{secPolymers}, we recall the even and odd polymer models introduced in~\cite{jenssen2020independent}, and state and extend some of the probabilistic estimates proved there.  

In Section~\ref{secGivenSize} we prove Theorems~\ref{thmFixedSize} and~\ref{thmFixedSizelam}, finding an expansion for a fugacity $\lam_\beta$ so that the expected size of an independent set is sufficiently close to $\lfloor \beta N \rfloor$ that a local central limit theorem will allow us to compute asymptotics. 

In Section~\ref{secLCLT} we show how local central limit theorems for polymer models follow from sufficiently fast convergence of the cluster expansion.  

Finally in Section~\ref{secSizeStructure} we combine  the above results to prove Theorem~\ref{thmFixedSizelocal}.

\section{The even and odd polymer models}
\label{secPolymers}

Let $\cE \subset V(Q_d)$ be the set of even vertices of the hypercube, those whose coordinates sum to an even number, and let $\cO \subset V(Q_d)$ be the odd vertices.   Note that $Q_d$ is a bipartite graph with bipartition $(\cE, \cO)$ and that $|\cE| = | \cO | =N:=2^{d-1}$. 
A key insight of~\cite{jenssen2020independent} is that, for $\lam$ not too small, the hardcore measure $\mu_{Q_d, \lam}$ can be closely approximated by a random perturbation of a random subset of either $\cO$ or $\cE$. 
The random perturbation takes the form of a polymer model with convergent cluster expansion, two notions that we introduce now. 

For a set $S \subseteq \cO$ (and analogously for $S \subseteq \cE$), let $|S|$ denote the number of vertices of $S$, $N(S)$ be the set of neighbors of $S$, and $[S]= \{ v \in \cO : N(v) \subseteq N(S) \}$ the bipartite closure of $S$. 
We call a set $S\subseteq \cO$ an \textit{odd polymer} if (i) the subgraph of $Q_d$ induced by the vertex set $S\cup N(S)$ is connected and (ii) $ |[ S]| \le N/2$. We let $\cP_\cO$ denote the set of all odd polymers. The \emph{weight} of an odd polymer $S$ is
\begin{align}\label{eqweightdef}
w(S)=\frac{\lam^{|S|}}{(1+\lam)^{|N(S)|}}\, .
\end{align}
We say that two odd polymers $S_1, S_2$ are \emph{compatible}, and write $S_1\sim S_2$, if the graph distance between $S_1, S_2$ is $Q_d$ is $>2$.
 We let $\Omega_\cO$ denote the set of all collections of mutually compatible odd polymers and define the following Gibbs measure on $\Omega_\cO$: for $\Gamma\in \Omega_\cO$,
 \[
 \nu_\cO(\Gamma)=\frac{\prod_{S\in \Gamma}w(S)}{\Xi_\cO}\, \quad \text{where} \quad \Xi_{\cO}= \sum_{\Gamma\in \Omega_{\cO}} \prod_{S\in \Gamma}w(S)
 \]
 is the odd polymer model partition function. Using $\nu_\cO$ we define a measure $\mu_{\cO, \lam}$ on independent sets in $Q_d$. 
\begin{defn}\label{defmu}
Let $\mu_{\cO, \lam}$ be the measure on $\cI$ defined by the following two-step process:
 \begin{enumerate}
 \item Choose a polymer configuration $\Gamma \in \Omega_{\cO} $ from $\nu_\cO$ and assign all vertices of $\cup_{S\in \Gamma} S$ to be occupied.\label{itemstep1}
\item For each vertex $v$ in $\cE$ that is not blocked by an occupied vertex in $\cO$, include $v$ in the independent set independently with probability $\frac{\lam}{1+\lam}$.  
\end{enumerate}
Let $Z_{\cO}(\lam) = (1+\lam)^N \Xi_{\cO}$, the independence polynomial of $Q_d$ restricted to independent sets achievable in the odd polymer model; that is, those for which $\mu_{\cO, \lam}$ assigns positive probability.  
 \end{defn}
 
 We think of Step 1 in Definition~\ref{defmu} as a perturbation of the `ground state' measure that simply selects a $p$-random subset of $\cE$ with $p=\lam/(1+\lam)$. 
The polymer configuration chosen in Step 1 will be typically small and so this process typically returns an independent set that is highly unbalanced with the majority of vertices even. We define even polymers, the even polymer model partition function $\Xi_\cE$, and measures $\nu_\cE$,  $\mu_{\cE, \lam}$ analogously. It was shown in~\cite{jenssen2020independent} that for $\lam>C\log d/d^{1/3}$, the hard core measure  $\mu_{Q_d, \lam}$ can be closely approximated by the mixture $\tfrac{1}{2}\mu_{\cO, \lam} + \tfrac{1}{2}\mu_{\cE, \lam}$.

\begin{theorem}[\cite{jenssen2020independent}]\label{thm:TV}
\label{lemPolyApprox}
For $\lam \ge C \log d/d^{1/3}$, we have
\begin{align}
\label{eqZXiapprox}
\left|  \log Z(\lam) - \log  \left[ 2 Z_{\cO} (\lam)\right ]  \right | = O\left( \exp(-N/d^4) \right ) \,.
\end{align}
  Moreover, letting $\hat\mu_{\lam}=\tfrac{1}{2}\mu_{\cO, \lam} + \tfrac{1}{2}\mu_{\cE, \lam}$, we have
 \begin{align*}
\|\mu_{\lam} - \hat\mu_{\lam} \|_{TV}  =O\left(\exp ( -N/d^4)\right) \,.
\end{align*}
Finally, with probability at least $1- O(\exp(-N/d^4))$ any defect vertices of $I$ drawn from $\mu_{\cO, \lam} $ are on the odd side of the bipartition; that is, the defects are the polymers of the polymer configuration. 
\end{theorem}

The lower bound on $\lam$ in Theorem~\ref{thm:TV} is an artifact of Sapozhenko's graph container method as implemented by Galvin~\cite{galvin2011threshold}. Theorem~\ref{thm:TV} quite possibly remains true for $\lam=\tilde\Omega(1/d)$ though proving this would require significant new ideas.

The power of Theorem~\ref{thm:TV} stems from the fact that the even and odd polymer models admit  convergent cluster expansions allowing for a detailed understanding of the measures $\mu_{\cO, \lam}, \mu_{\cE, \lam}$. Let us now introduce the cluster expansion formally.   

For a tuple $\Gamma$ of odd polymers, the \textit{incompatibility graph}, $H(\Gamma)$, is the graph with vertex set $\Gamma$ and an edge between any two incompatible polymers.  An odd \textit{cluster} $\Gamma$ is an ordered tuple of even polymers so that $H(\Gamma)$ is connected.  The size of a cluster $\Gamma$ is $\|\Gamma \| = \sum_{S \in \Gamma} |S|$.   Let $\cC$ be the set of all odd clusters.
For a cluster $\Gamma$ we define 
\begin{align*}
w(\Gamma)&= \phi (H(\Gamma)) \prod_{S \in \Gamma} w(S) \,,
\end{align*}
where $\phi(H)$ is the \textit{Ursell function} of a graph $H$, defined by
\begin{align}
\label{eqUrsell}
\phi(H) &= \frac{1}{|V(H)|!} \sum_{\substack{A \subseteq E(H)\\ \text{spanning, connected}}}  (-1)^{|A|} \, .
\end{align}
The cluster expansion is the formal infinite series
\begin{align}\label{eq:clusterexpansion}
\log \Xi_\cO= \sum_{\Gamma\in \cC} w(\Gamma)\, .
\end{align}
We define the cluster expansion of $\log \Xi_\cE$ analogously and note that by symmetry the expansions are identical.

In light of Theorem~\ref{lemPolyApprox}, throughout this section will we assume that $ \lam \ge C \log d/d^{1/3}$. We will also assume that $\lam=O(1)$ as $d\to\infty$.
 The following result from~\cite{jenssen2020independent} shows that for such $\lam$ the cluster expansion converges, we have good tail bounds on the expansion, and the terms of the cluster expansion can be efficiently computed. We say that a polynomial is \emph{computable in time $t$}, if its coefficients can be computed in time $t$. 

\begin{theorem}\label{thm:clusterexp}
For fixed $k \ge 1$,
\begin{align*}
 \sum_{\substack{\Gamma \in \mathcal{C} \\ \|\Gamma\| \geq k}}|w(\Gamma)|=  O \left ( \frac{2^d d^{2(k-1)}}{(1+ \lam)^{dk} }\right )
\end{align*}
and 
\begin{align}\label{eq:Rdef}
 \sum_{\substack{\Gamma \in \mathcal{C} \\ \|\Gamma\|= k}} w(\Gamma) = N\cdot  R_k(\lam ,d) (1+\lam)^{-kd}
\end{align}
where $R_k$ is a polynomial in $d$ and $\lam$ of degree at most $2k$ in $d$ and of degree at most $3k^2$ in $\lam$. Moreover $R_k$ is computable in time $e^{O(k \log k)}$. In particular,
\begin{align*}
\log \Xi_\cO 
&= N \sum_{j=1}^k  R_j(\lam ,d) (1+\lam)^{-jd}    +  O \left ( \frac{2^d d^{2k}}{(1+ \lam)^{d(k+1)} }\right )  \, .
\end{align*}

\end{theorem}
We note that Theorem~\ref{thmJPZ} follows in \cite{jenssen2020independent} from Theorems~\ref{lemPolyApprox} and~\ref{thm:clusterexp}. 

It will also be useful to record the following slightly strengthened tail bound on the cluster expansion. The following lemma is essentially contained in \cite{jenssen2020independent}, though it does not appear explicitly and so we provide the details. 

\begin{lemma}\label{lem:strongtail}
For $t, \ell\geq 1$ fixed,
\[
 \sum_{\substack{\Gamma \in \mathcal{C} \\ \|\Gamma\| \geq t}}|w(\Gamma)| \|\Gamma\|^\ell =  O \left ( \frac{2^d d^{2(t-1)}}{(1+ \lam)^{dt} }\right )\, .
\]
\end{lemma}
\begin{proof}
In \cite{jenssen2020independent} (see Lemma 15) it is shown that
\begin{align*}
\sum_{\Gamma \in \cC} |w(\Gamma)| e^{\gamma(d,\|\Gamma\|)} & \le 2^{d-1} d^{-3/2} \,,
\end{align*}
where
\begin{align*}
\gamma(d,k) &= \begin{cases}
\log(1+\lam)  (dk - 3 k^2)- 7 k\log d  \text{ if }  k \le \frac{d}{10}  \\
 \frac{d \log(1+\lam) k}{20} \text{ if }  \frac{d}{10} < k \le d^4\\
\frac{k}{d^{3/2}} \text{ if }  k > d^4 \,.
\end{cases} 
\end{align*}
Since $e^{\gamma(d,k)/2}\geq k^\ell$ for all $k$ and $d$ sufficiently large, it follows that

\begin{align*}
\sum_{\Gamma \in \cC} |w(\Gamma)|\|\Gamma\|^\ell e^{\gamma(d,\|\Gamma\|)/2} & =O( 2^{d} d^{-3/2}) \, .
\end{align*}
Keeping only terms in the above inequality corresponding to clusters of size at least $k$ we have
\begin{align}\label{eqgamtail}
 \sum_{\substack{\Gamma \in \mathcal{C} \\ \|\Gamma\| \geq k}} |w(\Gamma)|\|\Gamma\|^\ell  & \le O(2^{d} d^{-3/2}e^{-\gamma(d,k)/2}) \,.
\end{align}
With $t\geq0$ fixed we have by Theorem~\ref{thm:clusterexp} and~\eqref{eqgamtail} that 
\begin{align*}
\sum_{\substack{\Gamma \in \mathcal{C} \\ \|\Gamma\| \geq t}} |w(\Gamma)|\|\Gamma\|^\ell &=
 \sum_{\substack{\Gamma \in \mathcal{C} \\ t\leq\|\Gamma\| < 3t}} |w(\Gamma)|\|\Gamma\|^\ell 
 +
 \sum_{\substack{\Gamma \in \mathcal{C} \\ \|\Gamma\| \geq 3t}} |w(\Gamma)|\|\Gamma\|^\ell \\
 &= O \left ( \frac{2^d  d^{2(t-1)}}{(1+ \lam)^{dt} }\right ) + O\left(\frac{2^{d}d^{11t}}{(1+\lam)^{3dt/2}}\right)\\
 &= O \left ( \frac{2^d  d^{2(t-1)}}{(1+ \lam)^{dt} }\right ) \, .
\end{align*}
\end{proof}

Let $\mathbf \Gamma$ be a collection of compatible polymers sampled according to $\nu_{\cO}$ (the polymer measure at Step 1 of Definition~\ref{defmu}, the definition of $\mu_{\cO, \lam}$). We will use the above lemma to show that $\|\mathbf \Gamma\|$ and $|N(\mathbf \Gamma)|$ obey a central limit theorem. Formally, we say a sequence of random variables $(X_d)$ obeys a central limit theorem if $(X_d - \E X_d)/\sqrt{\var(X_d)}$ converges to $N(0,1)$ in distribution as $d\to\infty$. To prove this central limit theorem we will make use of a connection between the cluster expansion and \emph{cumulant generating functions}.

 Recall the cumulant generating function of a random variable $X$, is $h_t(X) = \log \E e^{tX}$. 
The $\ell$th \textit{cumulant} of $X$ is defined by taking derivatives of $h_t(X)$ and evaluating at $0$:
\begin{align*}
\kappa_\ell(X) &= \frac{ \partial^\ell h_t(X)}{\partial t^\ell} \Bigg|_{t=0} \,.  
\end{align*}
In particular, $\kappa_1(X)=\E (X)$ and $\kappa_2(X)=\var(X)$. 
The cumulants of $|N(\mathbf\Gamma)|$ can be expressed in terms of the cluster expansion as follows. 
Consider the odd polymer model with modified weights $w_t(S)=w(S)e^{t|N(S)|}$ for $t>0$ and let $\Xi_t$ denote the corresponding partition function. We then have 
\begin{align*}
h_t(|N(\mathbf \Gamma)|)= \log \Xi_t - \log \Xi_{\cO}\, .
\end{align*}
Applying the cluster expansion to  $\log \Xi_t$, taking derivatives, and evaluating at $t=0$ shows that
\begin{align}\label{eqclustercumulant}
\kappa_\ell(|N(\mathbf \Gamma)|)= \sum_{\Gamma\in\cC}w(\Gamma)|N( \Gamma)|^{\ell}\, .
\end{align}
Similarly $\kappa_\ell(\|\mathbf \Gamma\|)= \sum_{\Gamma\in\cC}w(\Gamma)\|\Gamma\|^{\ell}$.

\begin{lemma}\label{lemGamCLT}
Let $\mathbf \Gamma$ be a collection of compatible polymers sampled according to $\nu_{\cO}$. Then $\|\mathbf \Gamma\|$ and $|N(\mathbf \Gamma)|$ both obey a central limit theorem.
\end{lemma} 
\begin{proof}
We show that $|N(\mathbf \Gamma)|$ obeys a central limit theorem and the proof for $\|\mathbf \Gamma\|$ is identical. Let $Z=(|N(\mathbf \Gamma)| - \E |N(\mathbf \Gamma)|)/\sqrt{\var(|N(\mathbf \Gamma)|)}$. To show that $Z$ converges to $N(0,1)$ in distribution, it suffices to show that the cumulants of $Z$ converge to the cumulants of a standard normal i.e.\ it suffices to show that $\kappa_\ell(Z)\to 0$ for $\ell \geq 3$.
Now by~\eqref{eqclustercumulant} and Lemma~\ref{lem:strongtail},
\begin{align*}
\kappa_\ell(|N(\mathbf \Gamma)|)= \sum_{\Gamma\in \cC}w(\Gamma)|N(\Gamma)|^{\ell}\leq d^\ell \sum_{\Gamma\in \cC}w(\Gamma)\|\Gamma\|^{\ell}= O \left ( \frac{2^d d^\ell }{(1+ \lam)^{d} }\right )\, .
\end{align*}
On the other hand, by~\eqref{eqclustercumulant} and Lemma~\ref{lem:strongtail} again we have
\begin{align*}
\var(|N(\mathbf \Gamma)|)= \sum_{\Gamma\in \cC}w(\Gamma)|N(\Gamma)|^{2}\geq \sum_{\Gamma\in \cC}w(\Gamma)\|\Gamma\|^2= 2^{d-1}\frac{\lam}{(1+\lam)^d}+ O \left ( \frac{2^d  d^2}{(1+ \lam)^{2d} }\right )\, .
\end{align*}
It follows that if $\ell\geq 3$ then
\[
\kappa_\ell(Z)= \var(|N(\mathbf \Gamma)|)^{-\ell/2} \kappa_\ell(|N(\mathbf \Gamma)|)\to 0
\]
as desired.
\end{proof}

Next we will use the cluster expansion to give bounds on $\E_{\cO,\lam} ( | \mathbf I |)$, the expected size of an independent set sampled according to $\mu_{\cO, \lam}$. 
We begin by recording a useful expression for $\E_{\cO,\lam} ( | \mathbf I |)$ in terms of the cluster expansion.

\begin{lemma}\label{lem:expexpressions}
\begin{align*}
\E_{\cO,\lam} ( | \mathbf I |) 
&= \frac{\lam}{1+ \lam} N + \sum_{\Gamma\in\cC} w(\Gamma)\left(\|\Gamma\|- \frac{\lam}{1+\lam}|N(\Gamma)|\right).
\end{align*}
\end{lemma} 
\begin{proof}
Note that for an independent set $I$ such that $\mu_{\cO,\lam}(I)>0$, we have
\[
\mu_{\cO,\lam}(I)= \frac{\lam^{|I|}}{Z_\cO(\lam)}   = \frac{\lam^{|I|}}{(1+\lam)^N \Xi_\cO}\, .
\]
It follows that 
\[
\E_{\cO,\lam} ( | \mathbf I |)=\sum_I \frac{|I|\lam^{|I|}}{(1+\lam)^N \Xi_\cO}=\lam \frac{d}{d\lam} \log\left((1+\lam)^N \Xi_\cO \right)= \frac{\lam}{1+\lam}N + \lam(\log \Xi_\cO)'\, .
\]
We expand $\log \Xi_\cO$ via the cluster expansion as in~\eqref{eq:clusterexpansion} (which converges absolutely by Theorem~\ref{thm:clusterexp}). Recalling that 
$w(\Gamma)= \phi(H(\Gamma))\lam^{\|\Gamma\|}(1+\lam)^{-|N(\Gamma)|}$ 
for a cluster $\Gamma$,  we have
\begin{align*}
\E_{\cO,\lam} ( | \mathbf I |) 
&=\frac{\lam}{1+ \lam} N + \sum_{\Gamma\in\cC} \phi(H(\Gamma))  \frac{  \lam^{|\Gamma|}   ( (1+\lam)\|\Gamma\|  - \lam |N(\Gamma)|   ) }{  (1+\lam)^{|N(\Gamma)|+1}   } \\
&= \frac{\lam}{1+ \lam} N + \sum_{\Gamma\in\cC} w(\Gamma)\left(\|\Gamma\|- \frac{\lam}{1+\lam}|N(\Gamma)|\right)\, .
\end{align*}
\end{proof}

\begin{cor}\label{cor:expectation}
For fixed $k\geq 0$,
\begin{align*}
\E_{\cO,\lam} ( | \mathbf I |) &= \frac{\lam}{1+ \lam} N +  \lam N \sum_{j=1}^k \frac{ \frac{\partial}{\partial \lam}R_j(\lam,d)  }{ (1+\lam)^{jd}  }  - \frac{  jd R_j(\lam,d) }{ (1+\lam)^{jd+1}  }      +  O\left( \frac{2^d d^{2k+1}}{ (1+\lam)^{d(k+1)}} \right)\, .
\end{align*}
Where the $R_j(\lam, d)$ are as in Theorem~\ref{thm:clusterexp}.
\end{cor}
\begin{proof}
By Lemmas~\ref{lem:strongtail} and Lemma~\ref{lem:expexpressions}, noting that $|N(\Gamma)|\leq d\|\Gamma\|$ for any cluster $\Gamma$, we have
\begin{align*}
\E_{\cO,\lam} ( | \mathbf I |) 
= \frac{\lam}{1+ \lam} N +  \sum_{\substack{\Gamma \in \mathcal{C} \\ \|\Gamma\| \leq k}}  w(\Gamma)\left(\|\Gamma\|- \frac{\lam}{1+\lam}|N(\Gamma)|\right) + O \left ( \frac{2^d d^{2k+1}}{(1+ \lam)^{d(k+1)} }\right )\, .
\end{align*}
The result follows by recalling the definition of $R_j(\lam, d)$ at~\eqref{eq:Rdef}\, .
\end{proof}

\section{Independent sets of a given size}  
\label{secGivenSize}

Theorem~\ref{thmFixedSizelam} will follow from several lemmas.  The first says that almost all independent sets of size $m =  \lfloor \beta N \rfloor$ are accounted for by exactly one of the two polymer distributions.  The second says that if we find $\lam_\beta$ so that the expected size of the independent set drawn from $\mu_{\cO, \lam_\beta}$ (as in Definition~\ref{defmu}) is close to $m$ then we have an asymptotic formula for the number of independent sets of size $m$ in terms of $\lam_\beta$ and $Z_{\cO}(\lam_{\beta})$.  The third lemma gives an efficiently computable formula for a suitable such $\lam_\beta$.   We will then prove Theorem~\ref{thmFixedSize} by analyzing  expansions of  $\log Z_{\cO}(\lam_{\beta})$ and $\log \lam_\beta$ in  powers of $(1-\beta)^d$. 

  Let $i_m(\cO) $ be the number of independent sets $I$ of size $m$ in $Q_d$ that are achievable in the odd polymer model (i.e.\ $\mu_{\cO,\lam}(I)>0$).

\begin{lemma}
\label{lemimapprox}
For any $\beta >0$, 
\[ i_{\lfloor \beta N \rfloor} (Q_d) = (2+o(1)) i_{\lfloor \beta N \rfloor} (\cO) \]
as $d \to \infty$. 
\end{lemma}

\begin{lemma}
\label{lemImZ}
Fix $\beta>0$.  Suppose $\lam = \lam(\beta, d)$ is such that 
\begin{equation}
\label{eqLamClose}
\left | \E_{\cO,\lam} | \mathbf I |   -\lfloor \beta N \rfloor  \right|  =  o( N^{1/2})   \,.
\end{equation}
Then
\[ i_{\lfloor \beta N \rfloor}(\cO) = (1+o(1))  \frac{ (1+\lam)   Z_{\cO}(\lam)  }{  \lam^{\lfloor \beta N \rfloor}  \sqrt{2 \pi N \lam} }   \, .\]
\end{lemma}

\begin{lemma}
\label{lemPickLam}
There exists a sequence of rational functions $B_j(d,\beta)$, $j\in \mathbb N$, such that $B_j$ can be computed in time $e^{O( j \log j)}$ and the following holds.
Fix $t \ge 1$ and let $r=\lceil t/2 \rceil-1$. Suppose that $m = \lfloor \beta N \rfloor$ with $\beta > 1- 2^{-1/t} $, then if
\begin{equation}
\lam_{\beta} = \frac{\beta}{1-\beta} + \sum_{j=1}^r B_j(\beta, d) (1-\beta)^{jd}\, ,
\end{equation}
then
\[ \left | \E_{\cO,\lam_{\beta}} | \mathbf I |   -m  \right|  =  o( N^{1/2})   \,. \]
\end{lemma}

We prove Lemma~\ref{lemImZ} first, for which we need the following basic binomial local central limit result.
\begin{lemma}
\label{lemShiftBinomial}
Fix $p \in (0,1)$ and suppose $X \sim \text{Bin}(n,p)$. Suppose $n \to \infty$ and $k=o(\sqrt{n})$, then 
\[ \mathbb{P}(X= np +k) = \frac{1+o(1)}{\sqrt{ 2\pi n p(1-p)}} \,.\]
\end{lemma}

\begin{proof}[Proof of Lemma~\ref{lemImZ}]
Let $m = \lfloor \beta N\rfloor$. For any $\lam>0$, we have 
\begin{align*}
i_m(\cO) &=  \frac{ Z_{\cO}(\lam)}{\lam^m}  \mathbb{P}_{\cO,\lam} [ | \mathbf I| = m] \,,
\end{align*}
where the probability is with respect to the measure $\mu_{\cO, \lam}$.
We therefore need to show that if $\lam$ satisfies~\eqref{eqLamClose}, then
\begin{equation}
\label{eqProbBin1}
 \mathbb{P}_{\cO,\lam} [ | \mathbf I| = m]  = (1+o(1)) \frac{1+ \lam }{   \sqrt{2 \pi N \lam}  }   \,.
\end{equation}

By considering first the collection of polymers $\mathbf \Gamma$ chosen at Step 1 in the definition of $\mu_{\cO, \lam}$ (Definition~\ref{defmu}), and then the probability that the correct number of additional vertices are chosen at Step 2, we see that
\begin{align}\label{eqpolybinom1}
 \mathbb{P}_{\cO,\lam} [ | \mathbf I| = m] 
 =& \sum_{\Gamma\in\Omega_{\cO}}\mathbb{P}\left[\mathbf \Gamma = \Gamma \right] \cdot 
  \mathbb{P} \left [ \text{Bin}\left(N-|N( \Gamma)|, \frac{\lam}{1+\lam}\right) = m - \|\Gamma\|  \right]\, ,
\end{align}
where we recall that $\Omega_\cO$ denotes the set of all collections of mutually compatible odd polymers.
By the large deviation bound \cite[Lemma 16]{jenssen2020independent} we have
\[
\mathbb{P}\left[|N(\mathbf \Gamma)|\geq \frac{2N}{d}\right]=O( \exp(-N/d^4))\, ,
\]
and so we can condition on the event $|N(\mathbf \Gamma)|\leq \frac{2N}{d}$ throughout~\eqref{eqpolybinom1} and only change the resulting probability by an additive factor of $O( \exp(-N/d^4))=o(N^{-1/2})$. Under this conditioning, the binomial probabilities in~\eqref{eqpolybinom1} are uniformly bounded by $O(1/\sqrt{N})$. 
To establish~\eqref{eqProbBin1}, it therefore suffices to show that with high probability in the choice of $\mathbf \Gamma$, the binomial probabilities in~\eqref{eqpolybinom1} are in fact equal to  $(1+o(1)) \frac{1+ \lam }{   \sqrt{2 \pi N \lam}  }$.
By Lemma~\ref{lemShiftBinomial}, it suffices to show that with high probability in the choice of $\mathbf \Gamma$ we have
\begin{align}\label{eqbinmean}
\frac{\lam}{1+\lam}\left(N-|N( \Gamma)|\right)=m-\|\Gamma\| + o(N^{1/2})\, .
\end{align}
Now, by our assumption on $\lam$, \eqref{eqclustercumulant} and Lemma~\ref{lem:expexpressions}, 
\[
m= \E_{\cO,\lam} | \mathbf I | +  o(N^{1/2}) = \E \|\mathbf \Gamma\|+ \frac{\lam}{1+\lam}\left(N-\E |N(\mathbf \Gamma)|\right) + o(N^{1/2})\, .
\]
It follows that to show~\eqref{eqbinmean} holds whp with respect to $\mathbf \Gamma$, it suffices to show that
\begin{align}
\mathbb{P}\left[\|\mathbf \Gamma\|= \E \|\mathbf \Gamma\|+o(N^{1/2})\right]=1+o(1)\, ,
\end{align}
and similarly for $|N(\mathbf \Gamma)|$. This is an immediate consequence of Lemma~\ref{lemGamCLT} and the fact that 
\begin{align}\label{eqEGamoN}
\E\|\mathbf\Gamma\|\leq\E|N(\mathbf \Gamma)|=\sum_{\Gamma\in\cC}w(\Gamma)|N(\Gamma)|=o(N)\, ,
\end{align}
where we have used~\eqref{eqclustercumulant} and Lemma~\ref{lem:strongtail}.
\end{proof}

Next we prove Lemma~\ref{lemimapprox}.
\begin{proof}[Proof of Lemma~\ref{lemimapprox}]
Let $\mathcal I_m$ denote the set of all independent sets of size $m$ in $Q_d$.
Then by Theorem~\ref{thm:TV} (and the symmetry between even and odd) we have
\begin{align*}
\left|\mu_{\lam}(\mathcal I_m) - \hat\mu_{\lam}(\mathcal I_m) \right|
=
\left| \frac{i_m(Q_d)\lam^m}{Z_{Q_d}(\lam)} -  \frac{i_m(\cO)\lam^m}{ Z_{\cO}(\lam) } \right|  =O\left(\exp ( -N/d^4)\right)\, .
\end{align*}

By Theorem~\ref{thm:TV} again it follows that 
\begin{align*}
\left| i_m(Q_d)- (2+o(1))i_m(\cO)\right| =O\left(\exp ( -N/d^4)\right)\frac{ Z_{\cO}(\lam)}{\lam^m}\, .
\end{align*}

It therefore suffices to show that there is a choice of $\lam$ such that $\frac{i_m(\cO)\lam^m}{ Z_{\cO}(\lam)}$ is much larger than $\exp ( -N/d^4)$.  This follows from  Lemma~\ref{lemImZ}  by choosing $\lam$ satisfying~\eqref{eqLamClose}. 
\end{proof}

Next we prove Lemma~\ref{lemPickLam}. In the following, if $P(x,y)$ is a polynomial in $x,y$, we write $\deg_x(P)$ for the degree of $P$ in $x$.
\begin{proof}[Proof of Lemma~\ref{lemPickLam}]
Let $r=\lceil t/2 \rceil - 1$ and set
\[
\lam=\lam_{\beta} = \frac{\beta}{1-\beta} + \sum_{j=1}^r B_j(\beta, d) (1-\beta)^{jd}\, ,
\]
where the functions $B_j$ are rational polynomials in $\beta, d$ of constant degree (independent of $d$) to be determined later. Let $X:=\sum_{j=1}^r B_j(\beta, d) (1-\beta)^{jd+1}$ and note that $X=o(1)$. It will be useful to note that for $k=O(d)$,
\begin{align}\label{eqnegbinom}
(1+\lam)^{-k}=\frac{(1-\beta)^{k}}{(1+X)^{k}}
= (1-\beta)^{k}\sum_{i=0}^{r}\binom{-k}{i}X^i + O\left(X^{r+1}\right) \, .
\end{align}
In particular, since $\beta > 1-2^{-1/t}$,
\begin{align}\label{eqosqrtN}
(1+\lam)^{-d(r+1)}=(1+o(1))(1-\beta)^{d(r+1)}=O\left(e^{-cd} \cdot 2^{-d(r+1)/t}\right)=O\left(e^{-cd} N^{-1/2}\right)\, ,
\end{align}
for some constant $c>0$.

By Corollary~\ref{cor:expectation} and Theorem~\ref{thm:clusterexp},
\begin{align}
\E_{\cO,\lam} ( | \mathbf I |) 
&= N \frac{\lam}{1+ \lam}  +  \lam N \sum_{j=1}^r \frac{ \frac{\partial}{\partial \lam}R_j(\lam,d)  }{ (1+\lam)^{jd}  }  - \frac{  jd R_j(\lam,d) }{ (1+\lam)^{jd+1}  }      +  O\left(N \frac{d^{2r+1}}{ (1+\lam)^{d(r+1)}} \right)\nonumber \\
&=
\frac{\lam}{1+ \lam} N + N \sum_{j=1}^r F_j(\lambda, d)(1+\lam)^{-jd-1}   +  o\left(N^{1/2} \right)\, , \label{eqexpF}
\end{align}
where the $F_j$ are polynomials in $\lam, d$ with $\deg_d(F_j)\leq 2j+1$ and $\deg_\lam(F_j)\leq 3j^2$.
Our goal is to show that there exists an appropriate choice of $B_1, \ldots, B_r$ that makes this final expression~\eqref{eqexpF} equal to $m+o\left(N^{1/2} \right)$.

Since $\deg_\lam(F_j)\leq 3j^2$, and $\lam=(\beta+X)/(1-\beta)$ we may write
\[
F_j(\lam,d)=(1-\beta)^{-c_j}G_j(\beta, d, X)
\]
for some non-negative integer $c_j\leq 3j^2$ and $G_j$ a polynomial in $\beta, d, X$ such that $\deg_d(G_j)\leq 2j+1$.
It follows by~\eqref{eqnegbinom} that 
\begin{align}\label{eqXexp}
\nonumber \frac{\lambda}{1+\lambda} + \sum_{j=1}^r F_j(\lambda, d) (1+\lambda)^{-jd-1}&=
(\beta+X)\sum_{i=0}^{r}(-X)^i \\
&~~ +  \sum_{j=1}^rG_j(\beta, d, X) (1-\beta)^{jd+1-c_j}\sum_{i=0}^{r}\binom{-jd-1}{i}X^i + O(d^{3r}X^{r+1})\, .
\end{align}
We now recall that $X=\sum_{j=1}^r B_j (1-\beta)^{jd+1}$ and we expand this final expression as a polynomial in $(1-\beta)^d$. This yields 
\begin{align}\label{eqQexpansion}
\frac{\lambda}{1+\lambda} + \sum_{j=1}^r F_j(\lambda, d) (1+\lambda)^{-jd-1}&=
\beta + \sum_{j=1}^{r} Q_j(\beta, d, B_1,\ldots, B_r)\cdot (1-\beta)^{jd} + O(d^{3r}X^{r+1})
\end{align}
where $Q_j=Q_j(\beta, d, B_1,\ldots, B_r)$ is a rational function $\beta, d, B_1, \ldots, B_r$ with 
denominator $(1-\beta)^{b_j}$ for some $b_j\leq 3j^2$ and with $\deg_d(Q_j)\leq 2j+1$. 
Moreover, by examining the expansion~\eqref{eqXexp}, we see that $Q_j$  is linear in $B_j$ where the coefficient of $B_j$ is $(1-\beta)^2$ (in particular the coefficient is non-zero). It follows inductively that there is a choice of $B_1, \ldots, B_r$ such that $Q_1=\ldots=Q_r=0$ where $B_j$ is a rational function of $\beta, d$ of constant degree (depending on $j$ but not $d$). With this choice of $B_1, \ldots, B_r$ it follows from~\eqref{eqexpF} and \eqref{eqQexpansion} that 
\begin{align*}
\E_{\cO,\lam} ( | \mathbf I |) = \beta N+ O(Nd^{3r}X^{r+1})=\beta N+o(N^{1/2})\, ,
\end{align*}
where for the last bound we used that $d^{3r}X^{r+1}=d^{O_r(1)}(1-\beta)^{d(r+1)}=o(N^{-1/2})$ by~\eqref{eqosqrtN}.

Finally we note that the above argument gives an algorithm for computing the $B_j$. Since the $R_j$, and so also the $F_j$ and $G_j$, can be computed in $e^{O(j \log j)}$ time, we see that the $Q_j$ can be computed in $e^{O(j \log j)}$ time. The $B_j$ can then be computed by solving $j$ successive linear equations. 
\end{proof}

To give a concrete example of the algorithm above in action, we pause for a moment to calculate the rational function $B_1$. Using the definition of $R_1$ at~\eqref{eq:Rdef}, we see that $R_1=\lam$.
 In the notation of the proof of Theorem~\ref{lemPickLam}, it follows that $F_1=\lam+(1-d)\lam^2$. Noting that $\lam=(\beta+X)/(1-\beta)$ we have
\begin{align*}
F_1=(1-\beta)^{-2}\left[(1-\beta)(\beta+X) + (1-d)(\beta+X)^2 \right]
\end{align*}
and so $G_1= (1-\beta)(\beta+X) + (1-d)(\beta+X)^2$ and $c_1=2$.
Recalling that $X:=\sum_{j=1}^r B_j (1-\beta)^{jd+1}$ and examining the coefficient of $(1-\beta)^d$ in \eqref{eqXexp}, we see that
\begin{align*}
Q_1=B_1(1-\beta)^2 + \frac{\beta(1-\beta)+(1-d)\beta^2}{1-\beta}\, .
\end{align*}
Solving $Q_1=0$ yields 
\begin{align}\label{eqB1}
B_1=\frac{(d\beta-1)\beta}{(1-\beta)^3}\, .
\end{align}

\subsection{Proof of Theorem~\ref{thmFixedSizelam}}

Combining Lemma~\ref{lemimapprox}, Lemma~\ref{lemImZ}, and Lemma~\ref{lemPickLam} gives us the proof of Theorem~\ref{thmFixedSizelam}.

\subsection{Proof of Theorem~\ref{thmFixedSize}}

We now prove Theorem~\ref{thmFixedSize}.  Given the formula~\eqref{thmFixedSize} from Theorem~\ref{thmFixedSizelam}, we need to extract the binomial coefficient $\binom{N}{ \lfloor \beta N \rfloor}$ and expand the logarithm of what remains. 

\begin{lemma}
\label{lemBinomialCoeff}
Fix $\beta \in (0,1)$. With $\lam_0 =  \frac{\beta}{1 - \beta}$,
\begin{align*}
\binom{N}{ \lfloor \beta N \rfloor} &= (1+o(1)) \frac{ (1+\lam_0)^N} { \lam_0^{ \lfloor \beta N \rfloor}  \sqrt{2\pi N \beta (1-\beta)} }
\end{align*}
as $N \to \infty$. 
\end{lemma}
The proof follows from Stirling's formula.

\begin{proof}[Proof of Theorem~\ref{thmFixedSize}]
Given Lemma~\ref{lemBinomialCoeff}, Theorem~\ref{thmJPZ} and Theorem~\ref{thmFixedSizelam}, we are left to compute coefficients $P_j = P_j(\beta, d)$, $j \ge 1$, so that for $t \ge 1 $ and $\beta > 1- 2^{-1/t}$ we have
\begin{align}\label{eqPjexpansion}
 \log \left( \frac{1+\lam_\beta}{1+\lam_0}   \right) -   \beta  \log \frac{\lam_\beta}{\lam_0} +  \sum_{j=1}^{t-1} R_j(d,\lam_\beta) (1+\lam_\beta)^{-jd}    &=  \sum_{j=1}^{t-1} P_j \cdot (1-\beta)^{jd}  +o(N^{-1})
\end{align}
where $\lam_\beta$ is given by~\eqref{eqLamBetaDef}. 
We proceed by expanding each term on the left hand side of~\eqref{eqPjexpansion} 
as a power series in $(1-\beta)^d$. 
As in the proof of Lemma~\ref{lemPickLam} we set $X:=\sum_{j=1}^r B_j(\beta, d) (1-\beta)^{jd+1}$ where $r=\lceil t/2 \rceil-1$ and note that $X^t=o(N^{-1})$ since $\beta > 1- 2^{-1/t}$.
It follows by Taylor expansion that  
\begin{align}
\nonumber \log \left( \frac{1+\lam_\beta}{1+\lam_0}   \right) -   \beta  \log \frac{\lam_\beta}{\lam_0}
 &=
 \log(1+X)-\beta \log(1+X/\beta)\\
 &=
 \sum_{i=1}^{t-1}\frac{(-1)^{1+i}}{i}(1-\beta^{1-i})X^i + o(N^{-1})\label{eqlogexpansion}\, .
\end{align}
We now turn to the sum on the left hand side of~\eqref{eqPjexpansion}.
Since $R_j$ is a polynomial in $\lam_\beta, d$ such that 
$\deg_{\lam_\beta}(R_j)\leq 3j^2$ and $\deg_d(R_j)\leq 2j$
 and the fact that $\lam_\beta=(\beta+X)/(1-\beta)$, we may write
\[
R_j(\lam_\beta,d)=(1-\beta)^{-c_j}S_j(\beta, d, X)
\]
for some non-negative integer $c_j\leq 3j^2$ and $S_j$ a polynomial in $\beta, d, X$ such that $\deg_d(S_j)\leq 2j$.
It follows by~\eqref{eqnegbinom} that 
\begin{align}\label{eqRjexpansion}
\sum_{j=1}^t R_j(\lambda_\beta, d) (1+\lambda_\beta)^{-dj}&=
\sum_{j=1}^{t-1}S_j(\beta, d, X) (1-\beta)^{jd-c_j}\sum_{i=0}^{t-1}\binom{-jd}{i}X^i + O(d^{2t}X^t)\, .
\end{align}
We note that $O(d^{2t}X^t)=O(d^{3t}(1-\beta)^{td})=o(N^{-1})$.
To compute the $P_j$ we simply sum~\eqref{eqlogexpansion} and~\eqref{eqRjexpansion},
expand the powers of $X$, and collect the coefficients of $(1-\beta)^d, \ldots, (1-\beta)^{(t-1)d}$\, .
Finally we note that since $R_j, S_j$ and $B_j$ can each be computed in time $e^{O(j\log j)}$, $P_j$ can be computed in time $e^{O(j\log j)}$. 

\end{proof}

\subsubsection{Computation of $P_1, P_2$}
To illustrate the algorithm for computing the $P_j$ in Theorem~\ref{thmFixedSize}, 
we use it to compute $P_1$ and $P_2$. 
First we note that in~\cite{jenssen2020independent} it was shown that
\begin{align*}
R_1=\lam \text{\, and \, } R_2= \frac{(2\lam^3+\lam^4)d(d-1)-2\lam^2}{4}\, .
\end{align*}
It follows that (using the notation of the proof of Theorem~\ref{thmFixedSize})
\begin{align*}
S_1=\beta+X \text{\, and \, }
S_2=   \frac{1}{4} d(d - 1)  (2+ X -\b) (X+\b)^3 - \frac{1}{2} (1 - \b)^2 (X+\b)^2
\end{align*}
and $c_1=1$, $c_2=4$. 
Now, to compute $P_1, P_2$ we compute the coefficients of $(1-\beta)^d, (1-\beta)^{2d}$ in the sum of 
\eqref{eqlogexpansion} and~\eqref{eqRjexpansion}. 
The coefficient of $(1-\beta)^d$ in \eqref{eqlogexpansion} is $0$ and in \eqref{eqRjexpansion} it is
$\beta/(1-\beta)$ and so 
\begin{align*}
P_1=\frac{\beta}{1-\beta}\, .
\end{align*}
The coefficient of $(1-\beta)^{2d}$ in \eqref{eqlogexpansion} is 
\begin{align*}
\frac{1}{2\b}B_1^2(1-\beta)^3\, .
\end{align*}
The coefficient of $(1-\beta)^{2d}$ in \eqref{eqRjexpansion} is 
\begin{align*}
(1-d\b)B_1+(1-\b)^{-4}\left( \frac{1}{4} d(d - 1)  (2 -\b) \b^3 - \frac{1}{2} (1- \b)^2 \b^2 \right)\, .
\end{align*}
Recalling~\eqref{eqB1}, the formula for $B_1$, and summing the above two expressions yields
\begin{align*}
P_2= \frac{ d(d - 1)  ( 2- \b ) \b^3- 2(1- \b)^2 \b^2}{4(1-\b)^4}-\frac{\beta(1-d\b)^2}{2(1-\b)^3}\, .
\end{align*}

\section{Local central limit theorems for polymer models}
\label{secLCLT}

In the odd polymer model, let $T$ be a defect type, and $X_T$ be the random variable counting the number of defects of type $T$ in a sample from $\mu_{\cO, \lam}$.   Recall that $m_T = \E X_T$ and let $\sigma^2_T = \var(X_T)$. Moreover, let $n_T$ denote the number of polymers of type $T$ and let $w_T$ denote the weight $w(S)$ (defined at~\eqref{eqweightdef}) of a representative polymer $S$ of type $T$.

 Throughout this section we assume that $\lam \ge C\log d /d^{1/3}$ as in Theorem~\ref{lemPolyApprox} and probabilities and expectations are with respect to the odd polymer model.   The main result of this section is a multivariate local central limit theorem for the number of polymers of different types, extending the multivariate central limit theorem of~\cite[Theorem 6]{jenssen2020independent}.

\begin{theorem}
\label{thmMultiLCLTdefect}
Let $\cT_1$ and $\cT_2$ be two fixed sets of defect types so that for each $T \in \cT_1$, $m_T  \to \rho_T$ for some constant $\rho_T>0$, and for each $T \in \cT_2$, $m_T \to \infty$ as $d \to \infty$.  Let $ \{ k_T \}_{T \in \mathcal T_1}$ be a collection of non-negative integers and let $ \{ k_T \}_{T \in \mathcal T_2}$ be such that $k_T=  \lfloor m_T  +s_T \rfloor$ where $| s_T| = O(\sqrt{ m_T})$ for all $T \in \mathcal T_2$.  Then
\begin{align*}
\mathbb{P} \left(  \bigcap_{T \in \cT_1\cup \cT_2} X_T=k_T\right ) 
&= (1+o(1)) 
 \prod_{T \in  \cT_1} \frac{\rho_T^{k_T} e^{-\rho_T}}{(k_T)!}
\prod_{T \in \mathcal T_2} \frac{e^{ - \frac{ s_T^2}{2 m_T}}}{\sqrt{2 \pi m_T}}   \,.
\end{align*}
\end{theorem}
The probability in the theorem statement is with respect to the odd polymer model, but the statement also holds for defects of an independent set drawn from the hard-core model on $Q_d$ via Theorem~\ref{lemPolyApprox}.

Before we proceed it will be useful to recall a result from \cite{jenssen2020independent}
on the cumulants of the random variables $X_T$. 
Recall  that for a random variable $X$ we use $\kappa_k(X)$ to denote the $k$th cumulant of $X$. The following result appears as Lemma 20 in \cite{jenssen2020independent}.
\begin{lemma}
\label{lemYTkbounds}
For a defect type $T$, let $Y_T(\Gamma)$ denote the number of polymers of type $T$ in the cluster $\Gamma$.  Then for any fixed $k \ge 1$, 
\begin{equation}
\label{eqsumYk2}
 \kappa_k(X_T)= \sum_{\Gamma\in \cC} w(\Gamma) Y_T(\Gamma)^k = (1+o(1)) n_T w_T \,,
\end{equation}
and
\begin{equation}
\label{eqsumYk1}
 \sum_{\substack{\Gamma \in \mathcal{C} \\ \|\Gamma\| > |T|}}
 \left | w(\Gamma) Y_T(\Gamma)^k \right| = o( n_T w_T) \,.
\end{equation}
\end{lemma}

 Given a random vector $X=(X_1, \ldots, X_q)\in\R^q$ its characteristic function is
\[
\varphi_X(t)=\E e^{i\langle X, t \rangle}\, 
\]
for $t \in \R^q$. 

\begin{lemma}
\label{lemBigtmulti}
Fix $q\in\mathbb N$ and a list $T_1, \ldots, T_q$ of defect types. Let $X=(X_{T_1}, \ldots, X_{T_q})$.  There exists $c>0$ such that
\[
 |\varphi_{X}(t) | \le \exp\left\{- c \sum_{i=1}^q t_i^2 n_{T_i}w_{T_i}\right\} \, ,
 \]
for all $t\in[-\pi,\pi]^q$.
\end{lemma}

\begin{proof}
Given a cluster $\Gamma\in \cC$, let $Y(\Gamma)=(Y_{T_1}(\Gamma),\ldots, Y_{T_q}(\Gamma))$.
Using the cluster expansion we write
\begin{align*}
\log \E e^{i\langle t, X\rangle} &= \sum_{\Gamma\in \cC} w(\Gamma) e^{i \langle t, Y(\Gamma)\rangle}  - \sum_{\Gamma\in\cC} w(\Gamma) \\
&=  \sum_{\Gamma\in \cC} w(\Gamma) (e^{i \langle t, Y(\Gamma)\rangle}-1) \,.
\end{align*}
Let
\[
\mathcal{M}_j=\left\{\Gamma\in \cC: j=\max\{i: Y_{T_i}(\Gamma)>0\}, \sum_{i=1}^q Y_{T_i}(\Gamma)>1 \right\}\, .
\]
Then
\begin{align*}
\mathrm {Re} \log \E e^{i\langle t, X\rangle} &= \sum_{\Gamma\in \cC} w(\Gamma) (\cos(  \langle t, Y(\Gamma)\rangle)-1)  \\
&=\sum_{i=1}^q n_{T_i} w_{T_i} (\cos (t_i) -1)  +\sum_{j=1}^q\sum_{\Gamma\in \mathcal{M}_j} w(\Gamma) (\cos(  \langle t, Y(\Gamma)\rangle)-1)  \\
&\le-\frac{1}{5}\sum_{i=1}^q t_i^2 n_{T_i} w_{T_i}   +\sum_{j=1}^q\sum_{\Gamma\in \mathcal{M}_j} |w(\Gamma)| \langle t, Y(\Gamma)\rangle^2   \\
&\le-\frac{1}{5}\sum_{i=1}^q t_i^2 n_{T_i} w_{T_i}   +\sum_{j=1}^q\sum_{\Gamma\in \mathcal{M}_j} |w(\Gamma)|j\sum_{i=1}^jt_i^2Y_{T_i}(\Gamma)^2   \\
&\le  -\frac{1}{5}\sum_{i=1}^q t_i^2 n_{T_i} w_{T_i}   +\sum_{i=1}^qt_i^2o(n_{T_i}w_{T_i})   \\
\end{align*}
where for the first inequality we used that $-t^2\leq\cos(t)-1\leq -t^2/5$ for $t\in[-\pi,\pi]$. For the next inequality we used Cauchy-Schwarz, and for the final inequality we used Lemma~\ref{lemYTkbounds}.
\end{proof}

\begin{proof}[Proof of Theorem~\ref{thmMultiLCLTdefect}]
Let $\cT_1=\{T_1,\ldots, T_p\}$, $\cT_2=\{T_{p+1},\ldots, T_q\}$ and $\cT=\cT_1\cup\cT_2$. Let $X_i=X_{T_i}$, $m_i=m_{T_i}$, $\sigma_i=\sigma_{T_i}$, $k_i=k_{T_i}$ for $i\in [q]$. Let $X=(X_1,\ldots, X_q)$ and let 
\[
\tilde X=\left(X_1,\ldots, X_p, \frac{X_{p+1}-m_{p+1}}{\sigma_{p+1}}, \ldots,  \frac{X_{q}-m_{q}}{\sigma_{q}}\right)\, .
\]

By Fourier inversion,
\begin{align*}
\mathbb{P}\left(  \bigcap_{T \in \mathcal T} X_T=k_T\right) &= \frac{1}{(2 \pi)^q}  \int_{[-\pi,\pi]^q} \varphi_X(t) \cdot    e^{-i\langle t, k \rangle }  \, dt\, .
\end{align*}
Making the substitution $t_i=x_i$ for $i\in[p]$ and $t_i=x_i/\sigma_i$ for $i>p$ we have 
\begin{align}\label{eqxsub}
\mathbb{P}\left(  \bigcap_{T \in \mathcal T} X_T=k_T\right)  = \frac{1}{(2 \pi)^q} \prod_{i>p}\sigma_i^{-1}  \int_{B_1} \int_{B_2} \varphi_{\tilde X}(x)g(x)dx_q\ldots dx_1 
\end{align}
where $B_1=[-\pi,\pi]^p$, $B_2=[-\pi\sigma_{p+1}, \pi\sigma_{p+1}]\times\ldots\times [-\pi\sigma_{q}, \pi\sigma_{q}]$ and 
\[
g(x)=\exp\left\{i\sum_{j>p}x_j\frac{m_j-k_j}{\sigma_j}-i\sum_{j\leq p}x_jk_j\right\}\, .
\]

Let $Y=(Y_1, \ldots, Y_q)$ where $Y_{i}\sim \po(\rho_i)$ for $i\in [p]$, $Y_{i} \sim N(0,1)$ for $i>p$, and $Y_1,\ldots,Y_q$ are jointly independent.
Then by Fourier inversion we have the identity 
\begin{align*}
\frac{1}{(2 \pi)^q} \int_{B_1} \int_{\mathbb R^{q-p}} \varphi_{Y}(x)g(x)dx_q\ldots dx_1
 &=  \prod_{T \in  \cT_1} \frac{\rho_T^{k_T} e^{-\rho_T}}{(k_T)!}
\prod_{T \in \mathcal T_2} \frac{e^{- \frac{(k_T-m_T)^2}{2\sigma_T^2}}}{\sqrt{2\pi}}\, . \\
\end{align*}
By~\eqref{eqxsub} (noting that $m_i=(1+o(1))\sigma^2_i$ by Lemma~\ref{lemYTkbounds}), it therefore suffices to show that 
\begin{align*}
\int_{B_1} \int_{B_2} \varphi_{\tilde X}(x)g(x)dx_q\ldots dx_1
 = 
 \int_{B_1} \int_{\mathbb R^{q-p}} \varphi_{Y}(x)g(x)dx_q\ldots dx_1 + o(1)\, .
\end{align*}
Since, $m_i\to\infty$ for $i>p$, Lemma~\ref{lemYTkbounds} implies that $\sigma_i\to\infty$ for $i>p$ also. It follows that
\[
 \int_{B_1} \int_{\R^{q-p}\backslash B_2}\varphi_{Y}(x)g(x)dx_q\ldots dx_1 =o(1)\, 
\]
and so it suffices to show that 
\[
 \int_{B_1} \int_{B_2} |\varphi_{\tilde X}(x)-\varphi_Y(x)|dx_q\ldots dx_1 =o(1)\, .
\]

In~\cite[Theorem 6]{jenssen2020independent} it was shown that $\tilde X$ converges to $Y$ in distribution and so $\varphi_{\tilde X}\to \varphi_Y$ pointwise. 
It therefore suffices, by dominated convergence, to show that $|\varphi_{\tilde X}(x)-\varphi_Y(x)|$ is bounded by an integrable function. 
By Lemma~\ref{lemYTkbounds} and Lemma~\ref{lemBigtmulti},
\[
|\varphi_{\tilde X}(x)|= |\varphi_{X}(x_1,\ldots,x_p, x_{p+1}/\sigma_{p+1},\ldots, x_q/\sigma_q)|\leq e^{-\Theta(\sum_{i=1}^q x_i^2)}\, .
\]
We have a similar bound for $|\varphi_Y(x)|$ and so we are done.

\end{proof}

\section{Independent sets of a given size and structure}  
\label{secSizeStructure}

Here we prove Theorem~\ref{thmFixedSizelocal}.   Recall that for  a set of defect types $\cT$ and a vector of integers $ \mathbf x=( x_T )_{T \in \mathcal T}$, we let $ i_{m, \mathbf x}(Q_d) $ denote the number of independent sets in $Q_d$ of size $m$ with exactly  $ x_T $ defects of type $T$ for each $T \in \mathcal T$.

We  use the following identity, an easy extension of~\eqref{eqIkIdentity}. For any $\lam>0$, 
\begin{equation}
\label{eqIdenDefects}
i_{m, \mathbf x} (Q_d) = \frac{Z(\lam)}{\lam^m} \mathbb{P}_\lam \left[ |\mathbf I | =m,  (X_T)_{T \in \cT} = \mathbf x  \right ] \,,
\end{equation}
where $X_T$ be the random variable counting the number of defects of type $T$ in a sample from $\mu_{\lam}$.
Theorem~\ref{thmFixedSizelocal} follows immediately from~\eqref{eqIdenDefects}, Theorem~\ref{thmFixedSizelam}
and the following lemma.

\begin{lemma}
Fix $\lam>0$ and let $m = m(d)$ be such that that $|\E_{\lam} | \mathbf I| -m | = o(N^{1/2})$.     Let $\cT_1, \cT_2$ be the sets of defect types such that $m_T \to \rho_T$  for some fixed $\rho_T > 0$ as $d \to \infty$ for all $T\in \cT_1$ and $m_T\to\infty$ for all  $T\in \cT_2$. 
Let $ ( k_T )_{T \in \mathcal T_1}$ be a vector of fixed non-negative integers and let $ ( k_T )_{T \in \mathcal T_2}$ be such that $k_T=  \lfloor m_T  +s_T \rfloor$ where $| s_T| = O(\sqrt{ m_T})$ for all $T \in \mathcal T_2$. Let $\mathbf x= (k_T)_{T\in \cT_1\cup\cT_2}$. Then
\begin{align*}
\mathbb{P}_{\lam} \left [ |\mathbf I | =m,  (X_T)_{T \in \cT_1 \cup \cT_2} = \mathbf x  \right ] &= (1+o(1)) \mathbb{P}_{\lam} [ |\mathbf I | =m]  \prod_{T \in  \cT_1} \frac{\rho_T^{k_T} e^{-\rho_T}}{(k_T)!}
  \prod_{T \in \mathcal T_{2}} \frac{e^{ - \frac{ s_T^2}{2 m_T}}}{\sqrt{2 \pi m_T}} \,.
\end{align*} 
\end{lemma}

\begin{proof}

Using Theorem~\ref{lemPolyApprox} and Theorem~\ref{thmMultiLCLTdefect}, it is enough to show that 
\begin{align*}
\mathbb{P}_{\mathcal{O},\lam} \left[  | \mathbf I |= m  \Big |  (X_T)_{T \in \mathcal{T}_1 \cup \mathcal{T}_2} = \mathbf x   \right ] & = (1+o(1) ) \mathbb{P}_{\mathcal{O},\lam} \left[  | \mathbf I |= m \right ] \, ,
\end{align*}
where the above probabilities are with respect to $\mu_{\cO, \lam}$ and $X_T$ is now the random variable counting the number of defects of type $T$ in a sample from $\mu_{\cO, \lam}$.
Let $\mathbf \Gamma$ be the random collection of compatible polymers chosen at Step 1 in the definition of $\mu_{\cO, \lam}$ (Definition~\ref{defmu}) and let $\Gamma$ be a fixed collection of compatible polymers such that 
\begin{equation}
\label{eqn:cltgamma}
\|\Gamma\| = \mathbf{E}[\|\mathbf{\Gamma}\|] + o(N^{1/2})
\end{equation}
and
\begin{equation}
\label{eqn:cltngamma}
|N(\Gamma)| = \mathbf{E}[|N(\mathbf\Gamma)|] + o(N^{1/2}).
\end{equation}
Then the proof of Lemma~\ref{lemImZ} gives us that
\begin{align*}
\mathbb{P}_{\mathcal{O},\lam} \left[  | \mathbf I |= m  \Big | \mathbf{\Gamma} = \Gamma   \right ] = (1+o(1) ) \mathbb{P}_{\mathcal{O},\lam} \left[  | \mathbf I |= m \right ] \,
\end{align*}
and so in particular, if $\Gamma$ is also consistent with $\mathbf x$ (i.e.\ $\Gamma$ has precisely $k_T$ polymers of type $T$ for all $T\in \cT_1 \cup \cT_2$),
\begin{align*}
\mathbb{P}_{\mathcal{O},\lam} \left[  | \mathbf I |= m  \Big |  (X_T)_{T \in \mathcal{T}_1 \cup \mathcal{T}_2} = \mathbf x ,~\mathbf{\Gamma} = \Gamma  \right ] & = (1+o(1) ) \mathbb{P}_{\mathcal{O},\lam} \left[  | \mathbf I |= m \right ] \,. 
\end{align*}
Finally, Lemma~\ref{lemGamCLT} gives us that~(\ref{eqn:cltgamma}) and~(\ref{eqn:cltngamma}) both hold with probability $1 - o(1)$, completing the proof.

\end{proof}

\section*{Acknowledgements}
The authors thank Catherine Greenhill for helpful remarks about maximum entropy in combinatorial enumeration.  WP is supported in part by NSF grant DMS-1847451. AP is supported in part by NSF grant CCF-1934915.

\end{document}